\documentclass[a4paper,11pt]{article}
\usepackage{amsmath, amsfonts, amscd, amssymb, amsthm, enumerate, xypic}

\def\ad{\text{ad}}

\newcommand{\Mat}{\operatorname{M}}

\newcommand{\GL}{\operatorname{GL}}
\newcommand{\Ker}{\operatorname{Ker}}

\newcommand{\Vect}{\operatorname{span}}
\newcommand{\im}{\operatorname{Im}}

\newcommand{\tr}{\operatorname{tr}}

\newcommand{\rk}{\operatorname{rk}}
\newcommand{\codim}{\operatorname{codim}}
\renewcommand{\setminus}{\smallsetminus}
\newcommand{\modu}{\operatorname{mod}}

\def\K{\mathbb{K}}

\def\calL{\mathcal{L}}

\def\calS{\mathcal{S}}
\def\calT{\mathcal{T}}

\def\calV{\mathcal{V}}

\def\lcro{\mathopen{[\![}}
\def\rcro{\mathclose{]\!]}}

\theoremstyle{definition}

\theoremstyle{plain}
\newtheorem{theo}{Theorem}

\newtheorem{lemma}[theo]{Lemma}
\newtheorem{claim}{Claim}

\theoremstyle{plain}

\theoremstyle{remark}
\newtheorem{Rems}{Remarks}
\newtheorem{Rem}[Rems]{Remark}

\title{Lines of full rank matrices in large subspaces}
\author{Cl\'ement de Seguins Pazzis\footnote{Universit\'e de Versailles Saint-Quentin-en-Yvelines, Laboratoire de Math\'ematiques
de Versailles, 45 avenue des Etats-Unis, 78035 Versailles cedex, France}
\footnote{e-mail address: dsp.prof@gmail.com}}

\begin{document}


\thispagestyle{plain}

\maketitle

\begin{abstract}
Let $n$ and $p$ be non-negative integers with $n \geq p$, and $S$
be a linear subspace of the space of all $n$ by $p$ matrices with entries in a field $\K$.
A classical theorem of Flanders states that $S$ contains a matrix with rank $p$ whenever $\codim S <n$.

In this article, we prove the following related result: if $\codim S<n-1$, then, for any non-zero $n$ by $p$ matrix $N$
with rank less than $p$, there exists a line that is directed by $N$, has a common point with $S$ and contains only rank $p$ matrices.
\end{abstract}



\section{Introduction}

Throughout the article, $\K$ denotes an arbitrary field.
Let $n$ and $p$ be non-negative integers.
We denote by $\Mat_{n,p}(\K)$ the space of all $n$ by $p$ matrices with entries in $\K$.
In particular, we set $\Mat_n(\K):=\Mat_{n,n}(\K)$ and we denote by $\GL_n(\K)$ its group of units.
We denote by $E_{i,j}$ the matrix of $\Mat_{n,p}(\K)$ with zero entries everywhere except at the $(i,j)$-spot where the entry equals $1$.

In a landmark article \cite{Flanders}, Flanders proved the following classical result:

\begin{theo}[Flanders's theorem]
Let $n,p,r$ be non-negative integers such that $n \geq p \geq r$.
Let $S$ be a linear subspace of $\Mat_{n,p}(\K)$ in which every matrix has rank less than or equal to $r$.

Then, $\dim S \leq nr$.
\end{theo}

The upper-bound $nr$ is optimal, as shown by the example of the space of all matrices with zero entries in the last $p-r$ columns.
Before Flanders, Dieudonn\'e \cite{Dieudonne} had already studied spaces of singular square matrices and obtained
the special case $n=p$ and $r=n-1$ in the above theorem. Flanders actually had to assume that $\# \K>r$ due to his use of polynomials.
This provision was lifted by Meshulam \cite{Meshulam} (for more recent proofs, see \cite{affpres,dSPFlandersskew}).

Here is a reformulation of Flanders's theorem: if $n \geq p$, a linear subspace
$S$ of $\Mat_{n,p}(\K)$ such that $\dim S>nr$ must contain a matrix with rank greater than $r$.
In this work, we shall be concerned with not only finding one such matrix, but a whole
line of matrices with large rank. Better, we want to control the direction of such a line.

Before we formulate the problem, some basic considerations are necessary. Let
$N \in \Mat_n(\K) \setminus \{0\}$. If $N$ is invertible and $\K$ is algebraically closed, then every line directed
by $N$ must contain a singular matrix: indeed, for all $A \in \Mat_n(\K)$, we can write
$\forall \lambda \in \K, \; \det(A-\lambda N)=(-1)^n (\det N)\, p(\lambda)$
where $p$ denotes the characteristic polynomial of $N^{-1}A$, and $p$ must have a root.

Conversely, every non-zero matrix with non-full rank directs a line of full rank matrices, as
stated in the following lemma.

\begin{lemma}\label{fullspacelemma}
Let $n \geq p$ be non-negative integers and $N \in \Mat_{n,p}(\K)$ be such that
$\rk N<p$. Then, there exists $A \in \Mat_{n,p}(\K)$ such that every matrix of $A+\K N$ has rank $p$.
\end{lemma}

\begin{proof}
Set $r:=\rk N$.
Without loss of generality, we can assume that
$$N=\begin{bmatrix}
I_r & [0]_{r \times (p-r)} \\
[0]_{(n-r) \times r} & [0]_{(n-r) \times (p-r)}
\end{bmatrix}.$$
If $n>p$, one checks that $A:=\underset{j=1}{\overset{p}{\sum}} E_{j+1,j}$ has the requested property. \\
If $n=p$ one checks that the matrix
$A:=E_{1,n}+\underset{j=1}{\overset{n-1}{\sum}} E_{j+1,j}$ has the requested property.
\end{proof}

Now, here is our problem for square matrices:
given a linear subspace $S$ of $\Mat_n(\K)$ and a non-zero \emph{singular} matrix $N \in S$,
under what conditions on $\dim S$ can we guarantee that there exists $A \in S$ for which every matrix
of $A+\K N$ is invertible? More generally, if $n \geq p$, and given a linear subspace $S$ of $\Mat_{n,p}(\K)$ and a non-zero matrix
$N \in S$ with rank less than $p$, under what conditions on $\dim S$ can we guarantee that there exists $A \in S$ for which every matrix
of $A+\K N$ has rank $p$?

These questions are motivated by potential applications to the structure of spaces of bounded rank matrices over small finite fields.
The following theorem, which is the main point of the present article, gives a full answer to them.
\begin{theo}\label{rectangulartheorem}
Let $n \geq p \geq 2$ be integers. Let $S$ be a linear subspace of $\Mat_{n,p}(\K)$ with $\codim S\leq n-2$,
and let $N \in \Mat_{n,p}(\K)$ be such that $\rk N<p$.
Then, there exists $A \in S$ such that every matrix of $A+\K N$ has rank $p$.
\end{theo}

Here is a reformulation in terms of operator spaces:

\begin{theo}\label{operatortheorem}
Let $U$ and $V$ be finite-dimensional vector spaces with $\dim U \leq \dim V$.
Let $S$ be a linear subspace of $\calL(U,V)$ such that $\codim S \leq \dim V-2$, and $t \in \calL(U,V)$
be a non-injective operator. Then, there exists $a \in S$ such that every operator in $a+\K t$ is injective.
\end{theo}

Note, in the above theorems, that we do not require that the direction of the line be included in $S$!

Let us immediately show that the upper-bound $n-2$ from Theorem \ref{rectangulartheorem}
is optimal. Consider the matrix $N:=\begin{bmatrix}
I_{p-1} & [0]_{(p-1) \times 1} \\
[0]_{(n-p+1) \times (p-1)} & [0]_{(n-p+1) \times 1}
\end{bmatrix}$, and the space $S$ of all matrices of the form
$$\begin{bmatrix}
? & [?]_{1 \times (p-1)} \\
[0]_{(n-1)\times 1} & [?]_{(n-1) \times (p-1)}
\end{bmatrix}.$$
Then, for all $A \in S$, some matrix in $A+\K N$ has zero as its first column,
and hence not every matrix in $A+\K N$ has rank $p$.
Yet, $\rk N<p$ and $\codim S=n-1$.

\vskip 3mm
Theorem \ref{rectangulartheorem} will be proved in three steps.
In the first step, we shall consider the case of square matrices with $\rk N=n-1$.
The result actually deals with affine subspaces instead of just linear subspaces.

\begin{theo}\label{penciltheorem}
Let $n$ be a non-negative integer. Let $N$ be a rank $n-1$ matrix of $\Mat_n(\K)$.
Let $\calS$ be an affine subspace of $\Mat_n(\K)$ such that $\codim \calS \leq n-2$.
Assume that at least one matrix of $\calS$ maps $\Ker N$ into $\im N$.
Then, there exists $A \in \calS$ such that every matrix of $A+\K N$ is invertible.
\end{theo}

\begin{Rem}
Assume that $\K$ is algebraically closed.
Then, the condition that some matrix of $\calS$ maps $\Ker N$ into $\im N$ is unavoidable in Theorem \ref{penciltheorem}.
Consider indeed the matrix $N:=\begin{bmatrix}
I_{n-1} & [0]_{(n-1) \times 1} \\
[0]_{1 \times (n-1)} & 0
\end{bmatrix}$ and the affine hyperplane $\calS$ of all matrices of $\Mat_n(\K)$ with entry $1$ at the $(n,n)$-spot. For all $A \in S$, the polynomial $\det(A+tN)$ reads $t^{n-1}+\underset{k=0}{\overset{n-2}{\sum}} b_k t^k$,
and hence it is non-constant whenever $n\geq 2$, which yields that $A+\K N$ contains a singular matrix.
\end{Rem}

\begin{Rem}
If $\# \K>2$, the proof of Theorem \ref{penciltheorem} will actually demonstrate that there exists a matrix
$A \in \calS$ such that the (formal) polynomial $\det(A+tN)$ is constant and non-zero. As $\rk N=n-1$, this can be restated in terms of matrix pencils as saying that the matrix pencil $A+tN$ is equivalent to the pencil $I_n+t J$, where $J$ is the Jordan matrix $(\delta_{i,j-1})_{1 \leq i,j \leq n}$.

If $\# \K=2$, this result fails for $n=3$:
one considers the space $\calS$ of all matrices
of the form
$$\begin{bmatrix}
? & ? & a \\
? & ? & ? \\
? & a+1 & ?
\end{bmatrix} \quad \text{with $a \in \K$},$$
and the matrix
$$N:=\begin{bmatrix}
1 & 0 & 0 \\
0 & 1 & 0 \\
0 & 0 & 0
\end{bmatrix}.$$
One sees that $\calS$ has codimension $1$ in $\Mat_3(\K)$.
Let $M=\begin{bmatrix}
A & C \\
B & d
\end{bmatrix}\in \calS$, with $A \in \Mat_2(\K)$, $B \in \Mat_{1,2}(\K)$, $C \in \K^2$ and $d \in \K$. We have
\begin{align*}
\det(M+tN)& = d \det(A+tI_2)-B (A+t I_2)^\ad C \\
& =d \det(A+tI_2)+B (A^\ad+t I_2) C \\
& =d \det(A+tI_2)+t BC+B A^\ad C,
\end{align*}
where $A^\ad$ denotes the transpose of the matrix of cofactors of $A$. Assume that the polynomial $\det(M+tN)$ is constant.
As $\det(A+t I_2)$ has degree $2$, we successively obtain $d=0$ and $BC=0$. From the definition of $\calS$,
it follows that $B=0$ or $C=0$, and hence $\det(M+tN)=0$.

Finally, by checking the proof of Theorem \ref{penciltheorem}, one can prove that, if $\# \K=2$,
if $\codim \calS \leq n-3$ and some matrix of $\calS$ maps $\Ker N$ into $\im N$, then
$\det(A+tN)$ is constant and non-zero for some $A$ in $\calS$. We suspect that this result still holds, provided that $n>3$,
under the weaker assumption that $\codim \calS \leq n-2$.
\end{Rem}

In Section \ref{pencilproofsection}, Theorem \ref{penciltheorem} will be proved by induction over $n$.
In the next section, we shall extend it as follows, by considering an arbitrary singular matrix $N$.

\begin{theo}\label{squaretheorem}
Let $n$ be a non-negative integer. Let $N$ be a singular matrix of $\Mat_n(\K)$.
Let $\calS$ be an affine subspace of $\Mat_n(\K)$ such that $\codim \calS \leq n-2$.
Assume that there exists $M \in \calS$ such that the operator $X \in \Ker N \mapsto \overline{MX} \in \K^n/\im N$
is non-injective. Then, there exists $A \in \calS$ such that every matrix of $A+\K N$ is invertible.
\end{theo}

Again, this result will be proved by induction over $n$.

In the last step, by far the easiest one, we shall derive Theorem \ref{rectangulartheorem}
from Theorem \ref{squaretheorem} (see Section \ref{conclusionsection}).

The remaining open problem is the generalization of the above results to arbitrary ranks:
given non-negative integers $n,p,r$ such that $n \geq p \geq r$,
what is the smallest integer $d$ for which there exists
a matrix $N \in \Mat_{n,p}(\K)$ with rank less than $r$
and a linear subspace $S$ of $\Mat_{n,p}(\K)$ with codimension $d$
that contains no element $A$ for which all the matrices of $A+\K N$ have rank greater than or equal to $r$?
At the moment, we do not have a reasonable conjecture to suggest.

\section{Proof of Theorem \ref{penciltheorem}}\label{pencilproofsection}

The proof of Theorem \ref{penciltheorem} will be performed by induction over $n$, using several steps.
If $n\leq 1$ then the result is vacuous. If $n=2$, it is given by Lemma \ref{fullspacelemma}.
Assume now that $n \geq 3$.
We use a \emph{reductio ad absurdum}, by assuming that there is no matrix $A \in \calS$ such that
every matrix of $A+\K N$ is invertible.

Without loss of generality, we can assume that
$$N=\begin{bmatrix}
I_{n-1} & [0]_{(n-1) \times 1} \\
[0]_{1 \times (n-1)} & 0
\end{bmatrix}.$$
Then, we can split every matrix $M$ of $\Vect(\calS)$ up as
$$M=\begin{bmatrix}
A(M) & C(M) \\
L(M) & d(M)
\end{bmatrix}$$
with $A(M) \in \Mat_{n-1}(\K)$, $L(M) \in \Mat_{1,n-1}(\K)$, $C(M) \in \K^{n-1}$ and $d(M) \in \K$.
In $\calS$, we have the affine subspace
$$\calV:=\bigl\{M \in \calS : \; d(M)=0\bigr\}$$
with codimension at most $1$ (it is non-empty because we have assumed that at least one matrix of $\calS$ maps $\Ker N$ into $\im N$).
We denote by $V$ the translation vector space of $\calV$.
In $V$, we have two specific linear subspaces
$$T:=\{M \in V : \; L(M)=0 \; \text{and}\; C(M)=0\}$$
and
$$U:=\{M \in V : \; C(M)=0\}.$$
By the rank theorem, we have
\begin{equation}\label{ranktheorem}
\dim A(T)+\dim L(U)+\dim C(\calV)=\dim \calV.
\end{equation}
In particular, since $\dim \calV > n(n-1)$ and $\dim A(T) \leq (n-1)^2$ we find
\begin{equation}\label{dimC+dimL}
\dim C(\calV)+\dim L(U)>n-1.
\end{equation}
Given $X \in \K^{n-1} \setminus \{0\}$, we denote by $A(T)_X$ the linear subspace of $A(T)$
consisting of the matrices with column space included in $\K X$.
The bilinear form
$$b : (Y,X) \in \Mat_{1,n-1}(\K) \times \K^{n-1} \mapsto YX$$
is non-degenerate on both sides, and in the rest of the proof we shall consider orthogonality with respect to it.
Note in particular that \eqref{dimC+dimL} yields $C(\calV) \setminus L(U)^\bot \neq \emptyset$.

Note that, for all $P \in \GL_{n-1}(\K)$, neither the previous assumptions nor the conclusion are affected
in replacing $\calS$ with $Q \calS Q^{-1}$ where $Q:=P \oplus I_1$.
In this transformation the spaces $L(U)$ and $C(\calV)$ are respectively replaced with
$L(U)P^{-1}$ and $P C(\calV)$, whereas $b(YP^{-1},PX)=b(Y,X)$ for all $(Y,X) \in \Mat_{1,n-1}(\K) \times \K^{n-1}$.

\begin{claim}\label{claim1}
For all $X \in C(\calV) \setminus L(U)^\bot$, there exists
$M \in \calV$ such that $C(M)=X$ and $L(M)C(M)=0$.
\end{claim}

\begin{proof}
Let $X \in C(\calV) \setminus L(U)^\bot$.
We can find $(M_1,M_0) \in \calV\times U$ such that $C(M_1)=X$ and $L(M_0)X \neq 0$.
For all $\lambda \in \K$, we see that $C(M_1+\lambda M_0)=X$ and
$$L(M_1+\lambda M_0)C(M_1+\lambda M_0)=L(M_1)X+\lambda L(M_0)X,$$
and hence for a well-chosen $\lambda$ we find $L(M_1+\lambda M_0)C(M_1+\lambda M_0)=0$.
This proves our claim.
\end{proof}

\begin{claim}\label{claim2}
For all $X \in C(\calV) \setminus L(U)^\bot$, one has
\begin{equation}\label{diminequality1}
\dim C(\calV)+\dim A(T)_X \geq 2n-3.
\end{equation}
\end{claim}

\begin{proof}
We lose no generality in assuming that
$X=\begin{bmatrix}
1 \\
[0]_{(n-2) \times 1}
\end{bmatrix}$.
Denote by $\calV'$ the affine subspace of $\calV$ consisting of the matrices $M \in \calV$ such that
$C(M)=X$.
Every matrix $M \in \calV'$ splits up as
$$M=\begin{bmatrix}
[?]_{1 \times (n-1)} & 1 \\
K(M) & [0]_{(n-1) \times 1}
\end{bmatrix}$$
with
$$K(M)=\begin{bmatrix}
[?]_{(n-2) \times 1} & [?]_{(n-2) \times (n-2)} \\
? & [?]_{1 \times (n-2)}
\end{bmatrix}\in \Mat_{n-1}(\K).$$
Likewise, we write
$$N=\begin{bmatrix}
[?]_{1 \times (n-1)} & 0 \\
N' & [0]_{(n-1) \times 1}
\end{bmatrix}$$
with
$$N'=\begin{bmatrix}
[0]_{(n-2) \times 1} & I_{n-2} \\
0 & [0]_{1 \times (n-2)}
\end{bmatrix}.$$
By Claim \ref{claim1}, there exists $M \in \calV$ such that $C(M)=X$ and $L(M)X=0$, and hence $K(M)$ maps
$\Ker N'$ into $\im N'$. Moreover, $N'$ has rank $n-2$. Thus, if $\codim K(\calV') \leq n-3$, then by induction
we find a matrix $M \in \calV'$ such that $\det (K(M)+tN') \neq 0$ for all $t \in \K$; by developing the determinant along the last column,
it would follow that
$$\forall t \in \K, \; \det(M+tN)=(-1)^{n+1} \det(K(M)+tN')\in \K \setminus \{0\}.$$
This would contradict our assumptions.
Therefore, $\codim K(\calV') \geq n-2$.

However, by the rank theorem, we see that
$$\codim K(\calV')= \codim \calV+\bigl(\dim C(\calV)-(n-1))+\bigl(\dim A(T)_X-(n-1)\bigr).$$
Thus, as our assumptions yield that $\codim \calV \leq n-1$, we obtain claimed inequality \eqref{diminequality1}.
\end{proof}

It follows in particular that
\begin{equation}\label{minorCS}
\dim C(\calV) \geq n-2.
\end{equation}

\begin{claim}
One has $A(T)\subsetneq \Mat_{n-1}(\K)$.
\end{claim}

\begin{proof}
Assume on the contrary that $A(T)=\Mat_{n-1}(\K)$.

First, assume further that there exists $M \in \calV$ such that $L(M) \neq 0$, $C(M) \neq 0$ and $L(M)C(M)=0$.
As $A(T)=\Mat_{n-1}(\K)$, we can assume, without loss of generality, that
$$L(M)=\begin{bmatrix}
[0]_{1 \times (n-2)} & 1
\end{bmatrix}, \; C(M)=\begin{bmatrix}
1 \\
[0]_{(n-2) \times 1}
\end{bmatrix}\; \text{and} \;
A(M)=\begin{bmatrix}
[0]_{1 \times (n-2)} & 0 \\
I_{n-2} & [0]_{(n-2) \times 1}
\end{bmatrix}.$$
Then, it is easily checked that $\det(M+t N)=(-1)^{n+1}$, contradicting our basic assumptions on $\calV$.

Therefore,
\begin{equation}\label{keyimp}
\forall M \in \calV, \; L(M)C(M)=0
\Rightarrow (L(M)=0 \; \text{or}\; C(M)=0).
\end{equation}
Choose $X \in C(\calV) \setminus L(U)^\bot$. We know from Claim \ref{claim1} that there exists $M_1 \in \calV$ such that $C(M_1)=X$ and $L(M_1)X=0$.
Let $M_2 \in U$ be such that $L(M_2) \bot X$. Then, $C(M_1+M_2)=X$ and $L(M_1+M_2)=L(M_1)+L(M_2)$ is orthogonal to $X$.
It follows from \eqref{keyimp} that $L(M_1+M_2)=0$ and $L(M_1)=0$, whence $L(M_2)=0$. Therefore $L(U)\cap  \{X\}^\bot=\{0\}$, whence $\dim L(U)\leq 1$.
By inequality \eqref{dimC+dimL}, we deduce that $C(\calV)=\K^{n-1}$ and $\dim L(U)=1$.

From there, we split the discussion into two (non-disjoint) cases.

\begin{itemize}
\item \textbf{Case 1: $\# \K>2$.} \\
Let $M \in \calV$ be such that $C(M)\not\in L(U)^\bot$. We can choose $M_0 \in U$ such that $L(M_0)C(M) \neq 0$.
Then, for all $\lambda \in \K$, we have $C(M+\lambda M_0)=C(M)$ and $L(M+\lambda M_0) C(M+\lambda M_0)=L(M)C(M)+\lambda L(M_0)C(M)$;
we can then choose $\lambda \in \K$ such that $L(M+\lambda M_0) C(M+\lambda M_0)=0$, leading, by \eqref{keyimp}, to $L(M+\lambda M_0)=0$, and hence
$L(M)=L(-\lambda M_0) \in L(U)$. Hence, we have shown that $L(M) \in L(U)$ for all $M \in \calV$ such that $C(M)\not\in L(U)^\bot$.

Yet, as $L(U)^\bot$ is a proper affine subspace of $\K^{n-1}$,
its complementary subset in $\K^{n-1}$ generates the affine space $\K^{n-1}$
(remember that $\# \K>2$). Hence, $L(\calV) \subset L(U)$, leading to $\dim L(\calV) \leq 1$.
Then, by applying the same line of reasoning to $\calS^T$, which satisfies the same assumptions,
we would obtain $\dim C(\calV) \leq 1$, contradicting $C(\calV)=\K^{n-1}$ (remember that $n-1 \geq 2$).

\item \textbf{Case 2: $\K$ is finite.} \\
Then, we use a different strategy. Since $\dim L(U)=1$ and $\codim \calS \leq n-2$, we find a matrix $M_1 \in \calS$ such that $d(M_1) \neq 0$.
Since $C(\calV)=\K^{n-1}$, we also have $C(V)=\K^{n-1}$.
Hence, we can choose $M'_1 \in V$ such that $C(M'_1)=-C(M_1)$.
Hence, $M_2:=M_1+M'_1$ belongs to $\calS$ and satisfies
$d(M_2) \neq 0$ and $C(M_2)=0$.
As $n-1 \geq 2$ and $\K$ is a finite field, there exists a matrix $P \in \Mat_{n-1}(\K)$ with no eigenvalue:
it suffices to take $P$ as the companion matrix of an irreducible polynomial over $\K$ with degree $n-1$.
Since $A(T)=\Mat_{n-1}(\K)$, we can add a well-chosen matrix of $T$ to $M_3$ so as to find a matrix $M_3 \in \calS$ such that $d(M_3) \neq 0$, $C(M_3)=0$ and $A(M_3)=P$.
Then, $\det(M_3+t N)=d(M_3) \det(P+t I_{n-1}) \neq 0$ for all $t \in \K$, which contradicts our assumptions.
\end{itemize}

In any case, we have found a contradiction, which yields $A(T) \subsetneq \Mat_{n-1}(\K)$.
\end{proof}

Combining the previous claim with identity \eqref{ranktheorem} and $\dim \calV>n(n-1)$ yields
$$\dim C(\calV)+\dim L(U) > n.$$
In particular,
$$\dim L(U)\geq 2.$$

\begin{claim}\label{claimCStotal}
One has $C(\calV)=\K^{n-1}$.
\end{claim}

\begin{proof}
Assume on the contrary that $C(\calV)\subsetneq \K^{n-1}$. Then, $\dim C(\calV)=n-2$ by inequality \eqref{minorCS}.
We deduce from inequality \eqref{diminequality1} that, for all $X \in C(\calV)$,
the space $A(T)_X$ has dimension $n-1$, and hence it contains every matrix of $\Mat_{n-1}(\K)$ with column space $\K X$.
As $A(T)\subsetneq \Mat_{n-1}(\K)$, we deduce that $\Vect(C(\calV)) \subsetneq \K^{n-1}$, whence
$C(\calV)$ is a linear hyperplane of $\K^{n-1}$.

Next, let $Y_0 \in C(\calV)^\bot$. We claim that
$Y_0\, A(T) \subset \K Y_0$, that is $Y_0\, A(T) \bot C(\calV)$.
Let $X \in C(\calV) \setminus L(U)^\bot$. Let us prove that $Y_0 A(T) \bot X$. No generality is lost in assuming that
$$X=\begin{bmatrix}
1 \\
[0]_{1 \times (n-2)}
\end{bmatrix} \quad \text{and} \quad
Y_0=\begin{bmatrix}
[0]_{1 \times (n-2)} & 1
\end{bmatrix},$$
so that $C(\calV)=\K^{n-2} \times \{0\}$. As $\dim C(\calV)=n-2$ and $\codim A(T)>0$, inequality \eqref{ranktheorem} yields
$\dim L(U) \geq 3$. Then, we can find $M \in \calV$ such that $C(M)=X$, $L(M)X=0$ and $L(M) \notin \K Y_0$:
indeed, we know that we can find $M_1 \in \calV$ such that $C(M_1)=X$ and $L(M_1)X=0$ (see Claim \ref{claim1}).
Then, $L(U) \cap \{X\}^\bot$ has dimension at least $2$; we can choose $Z$ in $(L(U) \cap \{X\}^\bot) \setminus \K Y_0$;
then, we can choose $M_2 \in U$ such that $L(M_2)=Z$, and we check that
one of the matrices $M_1$ or $M_1+M_2$ must fulfill our needs.

Without further loss of generality, we can assume that $L(M)=\begin{bmatrix}
0 & 1 & [0]_{1 \times (n-3)}
\end{bmatrix}$.
Assume that there exists a matrix $J$ of $A(T)$ such that $Y_0 J$ is not orthogonal to $X$.
Then, for some $a \in \K \setminus \{0\}$, we have
$$J=\begin{bmatrix}
[?]_{(n-2) \times 1} & [?]_{(n-2) \times (n-2)} \\
a & [?]_{1 \times (n-2)}
\end{bmatrix}.$$
Since $A(T)$ contains every matrix with column space $\K X'$, for all $X' \in \K^{n-2} \times \{0\}$, we deduce that
there is a matrix $M'$ of $\calV$ such that $C(M')=X$, $L(M')=L(M)$ and
$$A(M')=\begin{bmatrix}
0 & 0 & [0]_{1 \times (n-3)} \\
[0]_{(n-3) \times 1} & [0]_{(n-3) \times 1} & I_{n-3} \\
a & ? & [?]_{1 \times (n-3)}
\end{bmatrix}$$
Then, one checks that $\det (M'+t N)=(-1)^{n-1} a$, which contradicts our assumptions.

Hence, $Y_0\, A(T) \bot X$ for all $X \in C(\calV) \setminus L(U)^\bot$. Since $\dim L(U) \geq 2$ and $\dim C(\calV)=n-2$, we find that
$L(U)^\bot \cap C(\calV)$ is a proper linear subspace of $C(\calV)$, and we conclude that
 $Y_0\, A(T) \bot C(\calV)$, as claimed.

Hence, $Y_0\, A(T) \subset \K Y_0$. In turn, this shows that $\codim A(T) \geq n-2$, and as
$\codim C(\calV)=1$ we deduce that $\codim \calV \geq n$, contradicting our assumptions.
\end{proof}

\begin{claim}\label{claimcodimAT}
One has $\codim A(T)=1$.
\end{claim}

\begin{proof}
Assume that such is not the case.
Let us consider the orthogonal $W$ of $A(T)$ for the non-degenerate symmetric bilinear form $(Z_1,Z_2) \mapsto \tr(Z_1Z_2)$ on $\Mat_{n-1}(\K)$.
Then, $\dim W \geq 2$.

The set $\widehat{W}:=\{Z \in W \mapsto ZX \mid X \in \K^{n-1}\}$ is a linear subspace of $\calL(W,\K^{n-1})$,
and we claim that every operator in it has rank at most $1$.
Assume that such is not the case. Then, we can find respective bases of $W$ and $\K^{n-1}$ in which one of the operators of
$\widehat{W}$ is represented by $\begin{bmatrix}
I_s & [0] \\
[0] & [0]
\end{bmatrix}$ for some integer $s \geq 2$. By assigning to every $X \in \K^{n-1}$ the determinant of the upper-left $2$ by $2$ submatrix
of the matrix representing $Z \mapsto ZX$ in the said bases, we define a non-zero quadratic form $q$ on $\K^{n-1}$
that vanishes at every vector $X \in \K^{n-1}$ such that $Z \in W \mapsto ZX$ has rank $1$.
For all $X \in \K^{n-1} \setminus L(U)^\bot$, we know that $\dim A(T)_X \geq n-2$ (see Claim \ref{claim2})
and hence $\rk(Z \in W \mapsto ZX) \leq 1$. Therefore, $q$ vanishes at every vector of $\K^{n-1} \setminus L(U)^\bot$.
Yet, $L(U)^\bot$ has codimension at least $2$ in $\K^{n-1}$.
Then, we deduce that $q=0$: if $\# \K>2$, this is easily obtained by choosing a non-zero linear form $\varphi$ on
$\K^{n-1}$ that vanishes everywhere on $L(U)^\bot$, and by noting that the homogenous polynomial $x \mapsto q(x)\varphi(x)$
with degree $3$ vanishes everywhere on $\K^{n-1}$; if $\# \K=2$ the statement follows directly from
Lemma 5.2 of \cite{dSPRC1}. This contradicts our assumptions.

Thus, $\widehat{W}$ is a linear subspace of $\calL(W,\K^{n-1})$ in which every operator has rank at most $1$.
As $\dim W>1$ and no vector of $W \setminus \{0\}$ is annihilated by all the operators in $\widehat{W}$,
the classification of vector spaces of rank $1$ operators
shows that there exists a $1$-dimensional linear subspace $D$ of $\K^{n-1}$ that includes the range of every operator in $\widehat{W}$,
which shows that $\im Z \subset D$ for all $Z \in W$.

Finally, as neither our assumptions nor our conclusion are modified in transposing both $N$ and $\calS$, we obtain that
the above property holds for $W^T$ as well, yielding a linear hyperplane $H$ of $\K^{n-1}$ such that
$H \subset \Ker Z$ for all $Z \in W$.
However, the space of all matrices $M \in \Mat_{n-1}(\K)$ such that $\im M \subset D$ and $H \subset \Ker M$ has dimension $1$,
contradicting the assumption that $\dim W \geq 2$.
\end{proof}

Now, we are about to conclude. We know that $C(\calV)=\K^{n-1}$ and that $L(U)^\bot$ is a proper linear subspace of $\K^{n-1}$
(since $\dim L(U)>0$). If, for all $X \in C(\calV) \setminus L(U)^\bot$, we had $\dim A(T)_X=n-1$,
it would follow that $A(T)=\Mat_{n-1}(\K)$, contradicting Claim \ref{claimcodimAT}.
Thus, we can find $X \in C(\calV) \setminus L(U)^\bot$ such that $\dim A(T)_X<n-1$.
As in the proof of Claim \ref{claimCStotal} (see its second paragraph), since $\dim L(U) \geq 2$ we can find
a matrix $M_1 \in \calV$ such that $C(M_1)=X$, $L(M_1)C(M_1)=0$ and $L(M_1) \neq 0$.
Without loss of generality we can assume that $X=\begin{bmatrix}
1 \\
[0]_{(n-2) \times 1}
\end{bmatrix}$ and $L(M_1)=\begin{bmatrix}
[0]_{1 \times (n-2)} & 1
\end{bmatrix}$. Now, as $\codim A(T)=1$ and $\dim A(T)_X<n-1$, the rank theorem yields that for every
$H \in \Mat_{n-2,n-1}(\K)$, there exists a matrix of $A(T)$ of the form
$\begin{bmatrix}
[?]_{1 \times (n-1)} \\
H
\end{bmatrix}$.
Thus, by adding a well-chosen matrix of $T$ to $M_1$, we reduce the situation to the one where
$$M_1=\begin{bmatrix}
[?]_{1 \times (n-2)} & ? & 1 \\
I_{n-2} & [0]_{(n-2) \times 1} & [0]_{(n-2) \times 1} \\
[0]_{1 \times (n-2)} & 1 & 0
\end{bmatrix}.$$
Then, one checks that $\det(M_1+t N)=(-1)^{n+1}$, which contradicts our initial assumptions.

This final contradiction shows that $\calS$ contains a matrix $M$ such that $\forall t \in \K, \; \det(M+tN) \neq 0$.
This completes the inductive proof.

\section{Proof of Theorem \ref{squaretheorem}}

We shall prove Theorem \ref{squaretheorem} by induction on $n$ and $r$.
Without loss of generality, we can assume that $N=\begin{bmatrix}
I_r & [0]_{r \times (n-r)} \\
[0]_{(n-r) \times r} & [0]_{(n-r) \times (n-r)}
\end{bmatrix}$ where $r:=\rk N$.
If $\calS=\Mat_n(\K)$ the result is known from Lemma \ref{fullspacelemma}.
In the rest of the proof, we assume that $\calS$ is a proper subspace of $\Mat_n(\K)$, and we denote by $S$ its translation vector space.

In particular, the case $n\leq 2$ is settled, and we assume that $n \geq 3$.
We perform a \emph{reductio ad absurdum}, by assuming that $\calS$ does not contain a matrix $A$ of the required form.
Theorem \ref{penciltheorem} gives the case when $r=n-1$.
In the rest of the proof, we assume that $r<n-1$.
We write every matrix $M$ of $\Mat_n(\K)$ as
$$M=\begin{bmatrix}
A(M) & C(M) \\
B(M) & D(M)
\end{bmatrix}$$
with $A(M) \in \Mat_r(\K)$, $B(M) \in \Mat_{n-r,r}(\K)$, $C(M) \in \Mat_{r,n-r}(\K)$ and $D(M) \in \Mat_{n-r}(\K)$.

The assumptions tell us that there exists $M_1 \in \calS$ such that $D(M_1)$ has rank less than $n-r$.
We distinguish between two cases.

\vskip 3mm
\noindent \textbf{Case 1: There exists a matrix $M_1 \in \calS$ such that $0<\rk D(M_1)<n-r$.} \\
Set $s:=\rk D(M_1)$.
By conjugating $\calS$ with a matrix of the form $I_r \oplus P$ for some well-chosen $P \in \GL_{n-r}(\K)$,
we see that no generality is lost in assuming that $D(M_1)=\begin{bmatrix}
[0] & [0] \\
[0] & I_s
\end{bmatrix}$.
Then, by applying row operations of the form $L_i \leftarrow L_i+\lambda L_n$ with $i \in \lcro 1,r\rcro$ and $\lambda \in \K$
and column operations of the form $C_j \leftarrow C_j+\mu C_n$ with $j \in \lcro 1,r\rcro$ and $\mu \in \K$,
no further generality is lost in assuming that the last row of $B(M_1)$ is zero and the last column of $C(M_1)$ is zero.

Denote by $\calS'$ the affine subspace of $\calS$ consisting of the matrices with the same last row as $M_1$.
Let us then write every matrix $M$ of $\calS'$ as
$$M=\begin{bmatrix}
K(M) & [?]_{(n-1) \times 1} \\
[0]_{1 \times (n-1)} & 1
\end{bmatrix} \quad \text{with $K(M) \in \Mat_{n-1}(\K)$.}$$
Then, with $N':=\begin{bmatrix}
I_r & [0]_{r \times (n-r-1)} \\
[0]_{(n-1-r) \times r} & [0]_{(n-1-r) \times (n-1-r)}
\end{bmatrix} \in \Mat_{n-1}(\K)$, we see that $K(M_1)$ is a matrix of $K(\calS')$ such that
$X \in \Ker N' \mapsto \overline{K(M_1) X} \in \K^{n-1}/\im N'$ has rank at most $n-2-r$
(as the first column of $D(M_1)$ is zero). If $\codim K(\calS')\leq n-3$, then by induction
we find that $K(\calS')$ contains a matrix $A'$ such that every matrix of $A'+\K N'$ is invertible:
writing $A'=K(A)$ for some $A \in \calS'$, we readily obtain that $\det(A+t N)=\det(A'+t N')$ for all $t$ in $\K$,
which yields that $A+tN$ is invertible for all $t \in \K$.
Hence, $\codim K(\calS')\geq n-2$, and as $\codim \calS \leq n-2$ we deduce from the rank theorem that
$S$ contains $E_{1,n},E_{2,n},\dots,E_{n-1,n}$.

Similarly, by considering the subspace of all matrices of $\calS$ with the same last column as $M_1$, we
find that $S$ contains $E_{n,1},\dots,E_{n,n-1}$.

Now, let $i \in \lcro 1,n-1\rcro$. Denote by $\calS_1$ the affine space deduced from
$\calS$ by the row operation $L_i \leftarrow L_i-L_n$ (which leaves $N$ invariant). As $\calS$ contains $M_1+E_{i,n}$, we
see that $\calS_1$ also contains $M_1$. Now, obviously $\calS_1$ satisfies all our assumptions with respect to $N$,
and it follows from our first step that the translation vector space of $\calS_1$ contains $E_{n,1},\dots,E_{n,n-1}$.
Hence, $S$ contains $E_{n,1}+E_{i,1},\dots,E_{n,n-1}+E_{i,n-1}$.
As $S$ also contains $E_{n,1},\dots,E_{n,n-1}$, we deduce that it contains $E_{i,1},\dots,E_{i,n-1}$.
Similarly, we obtain that, for all $j \in \lcro 1,n\rcro$, the space $S$ contains $E_{1,j},\dots,E_{n-1,j}$.
Hence, $S$ contains $E_{i,j}$ for all $(i,j)\in  \lcro 1,n\rcro^2 \setminus \{(n,n)\}$.
Then, the matrix $A:=E_{n,n}+E_{1,n-1}+\underset{i=1}{\overset{n-2}{\sum}} E_{i+1,i}$ belongs to $\calS$, and one checks that
the polynomial $\det(A+tN)$ is constant and non-zero, whence every matrix of $A+\K N$ is invertible. This contradicts our assumptions.

\vskip 3mm
\noindent \textbf{Case 2: For every matrix $R$ of $D(\calS)$, either $R=0$ or $R$ is invertible.} \\
Our assumptions then show that $D(\calS)$ contains $0$, and hence it is a linear subspace of $\Mat_{n-r}(\K)$.
Every matrix of $D(\calS)$ with first row zero equals zero, and hence $\dim D(\calS) \leq n-r$.

Now, denote by $\calT$ the affine subspace of $\calS$ consisting of its matrices $M$ such that $D(M)=0$.
For $M \in \calT$, let us write
$$C(M)=\begin{bmatrix}
C_1(M) & \cdots & C_{n-r}(M)
\end{bmatrix}.$$
If $C_1(\calT)=\{0\}$ then the rank theorem would yield $\codim \calS \geq r+(n-r)=n$,
contradicting our assumptions. Thus, there exists $M_1 \in \calT$ such that
$C_1(M_1) \neq 0$. Without loss of generality, we can assume that $C_1(M_1)=\begin{bmatrix}
1 \\
[0]_{(r-1) \times 1}
\end{bmatrix}$. Denote by $\calT'$ the space of all matrices of $\calT$ with the same $(r+1)$-th column as $M_1$.
For all $M \in \Mat_n(\K)$, we denote by $K(M)$ the submatrix of $M$ obtained by deleting the first row and the $(r+1)$-th column.
Assume that $\codim K(\calT') \leq n-3$. Then, the induction hypothesis applies to $K(\calT')$ and to $K(N')$:
indeed, every matrix of $K(\calT')$ maps $\Ker K(N)$ into $\im K(N)$,
and hence no such matrix induces an isomorphism from $\Ker K(N)$ to $\K^{n-1}/\im K(N)$ (because $n-1>r$).
Thus, we recover a matrix $M \in \calT'$ such that $K(M)+t K(N)$ is invertible for all $t$ in $\K$, and
as $\det(M+t N)=(-1)^r \det(K(M)+t K(N))$ for all $t \in \K$, we see that $M+t N$ in invertible for all $t \in \K$.

Hence, $\codim K(\calT) \geq n-2$. Yet, $\codim \calS \leq n-2$. By the rank theorem, it follows that $C_1(\calT)=\K^r$ and that
$S$ contains $E_{1,1},\dots,E_{1,r},E_{1,r+2},\dots,E_{1,n}$.

As $C_1(\calT)=\K^r$, we can apply the previous step to every non-zero vector of $\K^r$ rather than only to the first one of the standard basis.
It follows that $S$ contains $E_{i,j}$ for all $j \in \lcro 1,n\rcro \setminus \{r+1\}$ and all $i \in \lcro 1,r\rcro$.
With the same method applied to $C_k$, for all $k \in \lcro r+1,n\rcro$, we obtain that $S$ contains
$E_{i,j}$ for all $(i,j)\in \lcro 1,r\rcro \times \lcro 1,n-1\rcro$.

Now, by applying the previous step to $\calS^T$ we obtain that $S$ contains $E_{i,j}$ for all
$(i,j)\in \lcro 1,n\rcro \times \lcro 1,r\rcro$. Therefore, $\calT$ is the set of all $M \in \Mat_n(\K)$ such that $D(M)=0$.

We are about to conclude. As $\dim D(\calS) \leq n-r$ and $\codim \calS \leq n-2$, we see that $(n-r)(n-r-1) \leq n-2$.
Setting $s:=n-r$, we deduce that if $s >\frac{n}{2}$ then $\frac{n+1}{2}\,\frac{n-1}{2} \leq n-2$ (since $n>1$)
which would lead to $n^2-4n+7 \leq 0$, that is $(n-2)^2+3 \leq 0$. Therefore $s \leq \frac{n}{2}$, that is
$r \geq n-r$.
It follows that the matrix $A:=\underset{i=1}{\overset{r}{\sum}} E_{i,n-r+i}+\underset{j=1}{\overset{n-r}{\sum}} E_{r+j,j}$ belongs to $\calT$, and one checks that the polynomial $\det(A+t N)$ is constant and non-zero, whence every matrix of $A+\K N$ is invertible.

This completes our inductive proof of Theorem \ref{squaretheorem}.

\section{Proof of Theorem \ref{rectangulartheorem}}\label{conclusionsection}

We actually prove the ``operator space" version of Theorem \ref{rectangulartheorem}, that is Theorem \ref{operatortheorem}.
Once more, we use an induction over $\dim V$, with $U$ fixed.
Set $n:=\dim V$ and $p:=\dim U$.
The case $\dim U=\dim V$ is known by the operator space reformulation of Theorem \ref{squaretheorem}: in that case indeed the zero operator belongs
to $S$ and does not induce an injective operator from $\Ker t$ to $V/\im t$.
In the remainder of the proof, we assume that $\dim V>\dim U$.

Given a non-zero vector $y \in V$, we denote by $\pi_y : V \rightarrow V/\K y$ the canonical projection
and we set
$$S \modu y:=\{\pi_y \circ s \mid s \in S\},$$
which is a linear subspace of $\calL(U,V/\K y)$.

We perform a \emph{reductio ad absurdum}, by assuming that there is no operator $a \in S$
such that every operator of $a+\K t$ is injective.

Let $y \in V \setminus \{0\}$. Note that $\pi_y \circ t$ is non-injective.
We claim that $S \modu y$ contains no operator $a$ such that every operator in $a+\K(\pi_y \circ t)$
is injective: indeed, if such an operator $a$ existed, then $a=\pi_y \circ a'$ for some $a' \in S$, and hence, for all $\lambda \in \K$,
the operator $\pi_y \circ (a'+\lambda t)$ would be injective, which would show that $a'+\lambda t$ is injective.
By induction, we deduce that $\codim (S \modu y) > (\dim V-1)-2$
and hence $\codim (S \modu y)\geq \codim S$. It follows from the rank theorem that
$S$ contains every operator of $\calL(U,V)$ with range $\K y$. \\
Varying $y$ shows that $S=\calL(U,V)$,
and then Lemma \ref{fullspacelemma} yields a contradiction.

This completes the proof of Theorem \ref{rectangulartheorem}.

\end{document}